\documentclass[twoside]{amsart}
\usepackage{amsfonts,amsthm}
\usepackage{enumerate}
\usepackage{amsmath}
\usepackage{amssymb}
\usepackage{amsxtra}
\usepackage{latexsym}

\theoremstyle{plain}
\newtheorem*{thm A}{Theorem~A}
\newtheorem*{thm B}{Theorem~B}
\newtheorem*{thm C}{Theorem~C}

\newtheorem*{main 1}{Theorem~1}
\newtheorem*{main 2}{Theorem~2}
\newtheorem*{main 3}{Theorem~3}
\newtheorem*{coro 1}{Corollary~1}
\newtheorem*{coro 2}{Corollary~2}
\newtheorem*{lem 1}{Lemma~1}
\newtheorem*{lem 2}{Lemma~2}
\newtheorem*{lem 3}{Lemma~3}
\newtheorem*{pro A}{Proposition~A}
\newtheorem*{pro B}{Proposition~B}

\newtheorem*{note}{Note}

\newtheorem{theorem}{Theorem}[section]

\newtheorem{lemma}[theorem]{Lemma}

\newtheorem{remark}[theorem]{Remark}

\def \N{\nabla}

\def \al{\alpha}
\def \be{\beta}

\def \la{\lambda}

\def \si{\sigma}
\def \ka{\kappa}

\def \x{\xi}
\def \xo{{\xi}_1}
\def \xtw{{\xi}_2}
\def \xth{{\xi}_3}
\def \xt{\xi_{2}}
\def \xh{\xi_{3}}
\def \XN{\xi_{\nu}}

\def \p{\phi}
\def \po{\phi_{1}}

\def \PN{\phi_{\nu}}
\def \PX{\phi X}

\def \eo{\eta_1}
\def \et{\eta_{2}}
\def \eh{\eta_{3}}
\def \etw{\eta_{2}}
\def \ENx{{\eta}_{\nu}(\xi)}
\def \e{\eta}
\def \E{\eta}
\def \EN{{\eta}_{\nu}}

\def \ENK{{\eta}_{\nu}({\xi})}

\def \X{X_{0}}

\def \ENK{{\eta}_{\nu}({\xi})}
\def \PNK{{\phi}_{\nu}{\xi}}
\def \PNP{{\phi}_{\nu}{\phi}}

\def \RXP{R_\xi\phi}

\def \RX{R_\xi}
\def \RN{\Bar{R}_N}
\def \RNP{\Bar{R}_N\phi}
\def \RNX{\Bar{R}_N (X)}

\def \SN{\sum_{\nu=1}^3}

\def \GBt{G_2({\mathbb C}^{m+2})}
\def \GBo{G_2({\mathbb C}^{m+1})}

\def \QP{{\mathcal Q}^{\bot}}
\def \Q{\mathcal Q}

\begin{document}

\title[Commuting restricted Jacobi operators and Ricci tensor]{Real hypersurfaces in complex two-plane Grassmannians with commuting restricted \\Jacobi operators}

\vspace{0.2in}

\author[E. Pak, Y.J. Suh \& C. Woo]{Eunmi Pak, Young Jin Suh and Changhwa Woo}
\address{\newline Eunmi Pak, Young Jin Suh and Changhwa Woo
\newline Department of Mathematics,
\newline Kyungpook National University,
\newline Daegu 702-701, Republic of Korea}
\email{empak@knu.ac.kr} \email{yjsuh@knu.ac.kr}
\email{legalgwch@knu.ac.kr}


\thanks{* This work was supported by Grant Proj. No. NRF-2011-220-C00002 from National Research Foundation of Korea.
The first author by Grant Proj. No. NRF-2012-R1A2A2A01043023 and the third author is supported by NRF Grant funded by the Korean Government (NRF-2013-Fostering Core Leaders of Future Basic Science Program).}

\begin{abstract}
In this paper, we have considered a new commuting condition, that is, $(R_\xi\phi) S = S (R_\xi\phi)$ \big(resp. $(\Bar{R}_N\phi) S = S (\Bar{R}_N\phi$)\big) between the restricted Jacobi operator~$R_\xi\phi$ (resp. $\Bar{R}_N\phi$), and the Ricci tensor $S$ for real hypersurfaces $M$ in $G_2({\mathbb C}^{m+2})$. In terms of this condition we give a complete classification for Hopf hypersurfaces $M$ in $G_2({\mathbb C}^{m+2})$.
\end{abstract}

\maketitle

\section*{Introduction}
\setcounter{equation}{0}
\renewcommand{\theequation}{0.\arabic{equation}}
\vspace{0.13in}

The complex two-plane Grassmannians $\GBt$ are defined as the set of all complex two-dimensional linear subspaces in ${\mathbb C}^{m+2}$. It is a Hermitian symmetric space of rank~$2$ with compact irreducible type.  Remarkably, it is equipped with both a K\"{a}hler structure $J$ and a quaternionic K\"{a}hler structure ${\mathfrak J}$ (not containing $J$) satisfying $JJ_{\nu}=J_{\nu}J$ $(\nu=1,2,3)$, where $\{ J_{\nu}\}_{\nu=1, 2, 3}$ is an orthonormal basis of $\mathfrak J$. In this paper, we assume $m \geq 3$ (see Berndt and Suh~\cite{BS1} and~\cite{BS2}).
\vskip 5pt
Let $M$ be a real hypersurface in $\GBt$ and $N$ denote a local unit normal vector field to
$M$. By using the K\"{a}hler structure $J$ of $\GBt$, we can define a structure vector field by $\x=-JN$, which is said to be a {\it Reeb vector field}. If $\x$ is invariant under the shape operator $A$, it is said to be {\it Hopf}. In addition, $M$ is said to be a {\it Hopf hypersurface} if every integral curve of $M$ is totally geodesic. By the formulas in \cite[Section~$2$]{LCW}, it can be easily seen that $\x$ is Hopf if and only if $M$ is Hopf. From the quaternionic K\"{a}hler structure $\mathfrak J$ of $\GBt$, there naturally exist {\it almost contact 3-structure} vector fields defined by $\x_{\nu}=-J_{\nu}N$, $\nu=1,2,3$. Next, let us denote by $\QP=\text{Span}\{\,\xo, \xt, \xh\}$ a 3-dimensional distribution in a tangent space $T_{p}M$ at $p \in M$, where $\Q$ stands for the orthogonal complement of $\QP$ in $T_{p}M$.
Thus the tangent space of $M$ at $p\in M$ consists of the direct sum of $\Q$ and $\QP$, that is, $T_{p}M =\Q\oplus \QP$.
\vskip 5pt
For two distributions $[\x]=\text{Span}\{\,\x\}$ and $\QP$, we may consider two natural invariant geometric properties under the shape operator $A$ of~$M$, that is, $A [\xi] \subset [\xi]$ and $A\QP \subset \QP$. By using the result of Alekseevskii~\cite{Al-01}, Berndt and Suh~\cite{BS1} have classified all real hypersurfaces with these invariant properties in $\GBt$ as follows:
\begin{thm A}\label{thm A}
Let $M$ be a real hypersurface in $\GBt$, $m \geq 3$. Then both $[\x]$ and $\QP$ are invariant under the shape operator of $M$ if and only if
\begin{enumerate}[\rm(A)]
\item {$M$ is an open part of a tube around a totally geodesic $\GBo$ in $\GBt$, or} \
\item {$m$ is even, say $m = 2n$, and $M$ is an open part of a tube around a totally geodesic ${\mathbb H}P^n$ in $\GBt$}.
\end{enumerate}
\end{thm A}

\noindent In the case of~$(A)$ in Theorem~$A$, we want to say $M$ is of Type~$(A)$. Similarly, in the case of~$(B)$ in Theorem~$A$, we say $M$ is of Type~$(B)$.
\vskip 5pt
 Until now, by using Theorem~$A$, many geometers have investigated some characterizations of Hopf hypersurfaces in $\GBt$ with geometric quantities
 like shape operator, structure (or normal) Jacobi operator, Ricci tensor, and so on.
%there are various well-known results with respect to the Ricci tensor $S$ on Hopf hypersurfaces in $\GBt$
Commuting Ricci tensor means that the Ricci tensor $S$ and the structure tensor field $\p$ commute each other, that is, $S\p = \p S$.
From such a point of view, Suh~\cite{S02} has given a characterization of real hypersurfaces of Type~$(A)$ with commuting Ricci tensor%:
%\begin{thm B}\label{thm B-2}
%Let $M$ be a connected orientable Hopf hypersurface in $\GBt$ with commuting Ricci tensor, $m\geq 3$. Then $M$ is locally congruent to an open part of a tube around a totally geodesic
%$\GBo$ in $\GBt$.
%\end{thm B}
%\vskip 5pt
%Lee and Suh~\cite{LS} gave a characterization of real hypersurfaces of Type~$(B)$ in Theorem~$A$:
%\begin{thm C}\label{thm C}
%Let $M$ be a Hopf hypersurface in $\GBt$, $m \geq 3$. Then the Reeb vector field $\x$ belongs to the
%distribution $\Q$ if and only if $M$ is locally congruent to an open part of a tube around a totally geodesic ${\mathbb H}P^n$ in $\GBt$, $m=2n$, where the distribution $\Q$ denotes the orthogonal complement of $\QP$ in $T_{x}M$, $x \in M$. In other words, $M$ is locally congruent to of Type~$(B)$.
%\end{thm C}
\vskip 5pt
%\vskip 5pt
On the other hand, a Jacobi field along geodesics of a given Riemannian manifold $(\bar M, \bar g)$ is an important role in the study of differential geometry. It satisfies a well-known differential equation which inspires Jacobi operators. It is defined by $(\bar R_{X}(Y))(p)= (\bar R(Y, X)X)(p)$, where $\bar R$ denotes the curvature tensor of~$\bar M$ and $X$,~$Y$ denote any vector fields on $\bar M$. It is known to be a self-adjoint endomorphism on the tangent space $T_{p}\bar M$, $p \in \bar M$. Clearly, each tangent vector field $X$ to $\bar M$ provides a Jacobi operator with respect to $X$. Thus the Jacobi operator on a real hypersurface $M$ of $\GBt$ with respect to $\x$ (resp. $N$) is said to be a {\it structure Jacobi operator} (resp. {\it normal Jacobi operator}) and will be denoted by $R_{\xi}$ (resp. $\bar R_{N}$).
\vskip 5pt
For a commuting problem concerned with structure Jacobi operator $R_{\xi}$ and structure tensor $\p$ of $M$ in $\GBt$, that is, $R_{\xi} \p = \p R_{\xi}$, Suh and Yang~\cite{SY} gave a characterization of a real hypersurface of Type~$(A)$ in $\GBt$. Also, concerned with commuting problem for the normal Jacobi operator $\RN$, P\'erez, Jeong and Suh \cite{PJS} gave a characterization of a real hypersurface of Type~$(A)$ in $\GBt$.
\par
\vskip 5pt
On the other hand, another commuting problem $(\RXP) A =A (\RXP)$ \big(resp. $(\RNP) A =A (\RNP)$\big) related to the shape operator $A$ and the restricted structure Jacobi operator $\RXP$ \big(resp. the restricted normal Jacobi operator $\RNP$\big), which can be only defined in the orthogonal complement $[\x]^{\perp}$ of the Reeb vector field $[\x]$, was recently classified in \cite{LSW}.
\par
\vskip 5pt
Motivated by these results, let us consider the Ricci tensor $S$ instead of the shape operator $A$ for $M$ in $\GBt$. Then as a generalization, naturally, we consider a new commuting condition for the restricted structure Jacobi operator $\RXP$ and the Ricci tensor $S$ defined in such a way that
\begin{equation}\label{C-1}
(\RXP) S =S (\RXP).
\tag{C-1}
\end{equation}
The geometric meaning of \eqref{C-1} can be explained in such a way that any eigenspace of $R_{\x}$ on the distribution ${\mathfrak h}=\{X \in T_{x}M \mid X\perp \xi\}$, $x \in M$, is invariant by the Ricci tensor $S$ of $M$ in
$\GBt$. Now we want to give a complete classification of Hopf hypersurfaces in $\GBt$ with \eqref{C-1} as follows:
\begin{main 1}
Let $M$ be a Hopf hypersurface in complex two-plane Grassmannians $\GBt$, $m \geq 3$ with $(\RXP) S=S (\RXP) $.
If the smooth function $\al=g(A\x, \x)$ is constant along the direction of $\x$, then $M$ is locally congruent with an open part of a tube of some
radius $r \in (0,\frac{\pi}{2\sqrt{2}})$ around a totally geodesic $\GBo$ in $\GBt$.
\end{main 1}
Next, we want to consider another commuting condition between the restricted normal Jacobi operator $\RNP$ and the Ricci tensor $S$ defined by
\begin{equation}\label{C-2}
(\RNP) S =S (\RNP),
\tag{C-2}
\end{equation}
and give a classification of Hopf hypersurfaces in $\GBt$ with \eqref{C-2} as follows:
\begin{main 2}
Let $M$ be a Hopf hypersurface in complex two-plane Grassmannians $\GBt$, $m
\geq 3$ with $(\RNP) S =S (\RNP)$. If the smooth function $\al=g(A\x, \x)$ is constant along the
direction of $\x$, then $M$ is locally congruent to an open part of a tube of some
radius $r \in (0,\frac{\pi}{2\sqrt{2}})$ around a totally geodesic $\GBo$ in $\GBt$.
\end{main 2}
\vskip 5pt
Actually, according to the geometric meaning of the condition \eqref{C-1}(resp. \eqref{C-2}), we also assert that any eigenspaces of the Ricci tensor~$S$ on $M$ in $\GBt$ are invariant under the restricted structure Jacobi operator $\RXP$ (resp. the restricted normal Jacobi operator $\RNP$). In Sections~\ref{Section 1} and \ref{Section 2}, we give a complete proof of Theorems~$1$ and $2$, respectively. We refer to \cite{Al-01}, \cite{BS1}, \cite{BS2} and \cite{LS} for Riemannian geometric structures of $\GBt$, $m \geq 3$.
\vskip 5pt
%%%%%%%%%%%%%%%%%%%%%%%%%%%%%%%%% 3333333333333 %%%%%%%%%%%%%%%%%%%%%%%%%%%%%%%%%%%%
\section{Proof of Theorem~$1$}\label{Section 1}
\setcounter{equation}{0}
\renewcommand{\theequation}{1.\arabic{equation}}
\vspace{0.13in}

%The Ricci tensor $S$ is defined by the trace of the Riemann curvature tensor.
%It is the most standard way to express curvature of Riemannian manifolds and a tensor field
%that measures the extent to which the metric tensor is not locally isometric to a Euclidean space (see \cite{Besse}).
In this section, by using geometric quantities in \cite{S02} and \cite{SY}, we give a complete proof of Theorem~$1$.
To prove it, we assume that $M$ is a Hopf hypersurface in $\GBt$ with \eqref{C-1}, that is,
\begin{equation}\label{eq: 1.1}
(\RXP) SX = S (\RXP)X.
\end{equation}
From now on, $X$,$Y$ and $Z$ always stand for any tangent vector fields on $M$.

Let us introduce the Ricci tensor~$S$ and structure Jacobi operator~$R_{\x}$, briefly.
The curvature tensor $R(X,Y)Z$ of $M$ in $\GBt$ can be derived from the curvature tensor ${\bar R}(X,Y)Z$
of $\GBt$. Then by contracting and using the geometric structure $JJ_{\nu}=J_{\nu}J$ $(\nu=1,2,3)$ related to the K\"{a}hler structure $J$ and the quaternionic K\"{a}hler structure $J_{\nu}$ $(\nu=1,2,3)$, we can derive the Ricci tensor $S$ given by
$$g(SX,Y)={\sum}_{i=1}^{4m-1}g(R(e_i,X)Y,e_i),$$
where $\{e_1, {\cdots}, e_{4m-1}\}$ denotes a basis of the tangent space $T_xM$ of $M$, $x{\in}M$, in $\GBt$ (see~\cite{S02}).

\noindent From the definition of the Ricci tensor~$S$ and fundamental formulas in \cite[section~$2$]{S02}, we have
\begin{equation} \label{eq: 1.2}
\begin{split}
SX & = \sum_{i=1}^{4m-1} R(X,e_i)e_i \\
   & = (4m+7)X - 3\eta(X)\x + hAX - A^2 X \\
   & \quad + \sum_{\nu =1}^3 \{ - 3 \eta_{\nu}(X)\x_{\nu} +
   \eta_{\nu}(\x)\p_{\nu}\p X-\eta(\p_{\nu}X)\p_{\nu}\x -\eta(X)\eta_{\nu}(\x)\x_{\nu} \},
\end{split}
\end{equation}
where $h$ denotes the trace of $A$, that is, $h=\text{Tr}A$ (see \cite[$(1.4)$]{PS}).
By inserting $Y=Z=\x$ into the curvature tensor $R(X,Y)Z$ and using the condition of being Hopf,
the structure Jacobi operator $R_{\x}$ becomes
\begin{equation} \label{eq: 1.3}
\begin{split}
R_{\x}(X)&=R(X,{\x}){\x}\\
&= X - {\eta}(X){\x} - {\SN}\Big\{{\EN}(X){\XN}-{\eta}(X){\ENK}{\XN}\\
&\ \ +3g({\PN}X,{\x}){\PNK}+{\ENK}{\PNP}X\Big\}+{\alpha}AX-\al^2{\eta}(X){\x}
\end{split}
\end{equation} (see \cite[section~$4$]{JMPS}).

\vskip 5pt
Using these equations \eqref{eq: 1.1}, \eqref{eq: 1.2} and \eqref{eq: 1.3}, we prove that the Reeb vector field $\x$ of $M$ belongs to either $\Q$ or $\QP$.
\begin{lemma}\label{lemma 1.1}
Let $M$ be a Hopf hypersurface in $\GBt$, $m \geq 3$, with \eqref{C-1}. If the principal curvature $\alpha=g(A\x, \x)$ is constant along the direction of $\x$, then $\x$ belongs to either the distribution~$\Q$ or the distribution~$\QP$.
\end{lemma}

\begin{proof}
In order to prove this lemma, we put
\begin{equation}\label{**}
\x = \eta(X_{0})X_{0}+\eta(\x_{1})\x_{1}
\end{equation}
for some unit vectors $X_{0} \in \Q$, $\x_{1} \in \QP$ and $\eta(X_{0})
\eta(\x_{1})\neq 0$.
\vskip 5pt
In the case of $\alpha=0$, by virtue of $Y{\alpha}=({\x}{\alpha}){\eta}(Y)-4{\SN}{\ENK}{\EN}({\p}Y)$ in \cite[Lemma~$1$]{BS1}, we obtain easily that $\x$ belongs to either $\Q$ or $\QP$.\\
\indent Thus, we consider the next case~$\al \neq 0$.
Putting $X=\x$ in \eqref{eq: 1.1} and using the fact $\p\x=0$, it follows that
\begin{equation}\label{eq: 1.5}
(\RXP) S\x = 0.
\end{equation}

From \eqref{eq: 1.2} and \eqref{**}, we have
\begin{align}
& S\p\X=\sigma\p\X,\label{eq: 1.6} \\
& S\X=(4m+7+h\al-\al^2)\X+\eo^2(\x)\X-\e(\X)\X,\label{eq: 1.7}\\
& S\x=(4m+4+h\al-\al^2)\x-4\eo(\x)\xo,\label{eq: 1.8}
\end{align}
where $\sigma:=4m+8+h\ka+\ka^2$.\\
\noindent Multiplying $\p$ to \eqref{eq: 1.8}, we have
\begin{equation}\label{eq: 1.9}
\p S\x=-4\eta(\x_{1})\p\x_{1}.
\end{equation}
\indent From $\p \x =0$, we obtain $\po\x=\e(\X)\po\X$ and $\p\X=-\e(\xo)\po\X$. Because of $\e(\X)\e(\xo)\neq0$ and \eqref{eq: 1.9}, \eqref{eq: 1.5} becomes
\begin{equation}\label{eq: 1.10}
\begin{split}
0=\RX(\p\x_{1})=\RX(\po\X)=\RX(\p \X).
\end{split}
\end{equation}
\indent By substituting $X=\p \X$ into \eqref{eq: 1.3} and using \eqref{eq: 1.10}, we get

\begin{equation}\label{eq: 1.11}
 A\p X_0=-\frac{4\e^2(X_0)}{\al}\p \X.
\end{equation}

Due to \cite[Equation~$(2.10)$]{JMPS}, $A \xo = \alpha \xo$ is derived from $\xi \alpha =0$.
This leads to
\begin{equation}\label{eq: 1.12}
A\p\X=\ka\p\X,
\end{equation}
where $\ka=\frac{\al^2+4\e^2{\X}}{\al}$ (see \cite[section~$4$]{JMPS}).

Combining \eqref{eq: 1.11} and \eqref{eq: 1.12}, we obtain
\begin{equation*}
\{\al^2+8\e^2(\X)\}\p\X=0.
\end{equation*}
\indent This means $\p\X=0$ which gives rise to a contradiction. Thus this lemma is proved.
\end{proof}

Now, we shall divide our consideration into two cases that $\x$
belongs to either $\QP$ or $\Q$,
respectively. Next, we further study the case $\x \in \QP$. We may put $\x=\x_{1}\in \QP$ for our convenience sake.
\begin{lemma}\label{lemma 1.2}
Let $M$ be a Hopf hypersurface in $\GBt$. If the Reeb vector field $\x$ belongs to $\QP$, then the Ricci tensor~$S$ commutes with the shape operator~$A$, that is, $SA=AS$.
\end{lemma}
\begin{proof}

Differentiating $\x=\xi_1$ along any direction $X \in TM$ and using \cite[section~$2$, $(2.2)$ and $(2.3)$]{LKS01}, it gives us
\begin{equation}\label{eq: 1.13}
\phi AX=\N_{X}\x=\N_{X}\xo=q_{3}(X)\xt-q_{2}(X)\xh+\phi_{1} AX.
\end{equation}
\indent Taking the inner product with $\xt$ and $\xh$ in \eqref{eq: 1.13}, respectively gives $q_{3}(X)=2\eta_{3}(AX)$ and $q_{2}(X)=2\eta_{2}(AX)$.
Then \eqref{eq: 1.13} can be revised:
\begin{equation}\label{eq: 1.14}
\p AX=2\eta_{3}(AX)\xt-2\eta_{2}(AX)\xh+\po AX.
\end{equation}
%(see \cite[Lemma~$3.2$]{LKS01}).\\
\indent From this, by applying the inner product with any tangent vector $Y$, we have
\begin{equation*}
g(\p AX,Y)=2\eh(AX)g(\xt,Y)-2\et(AX)g(\xh,Y)+g(\po AX,Y).
\end{equation*}
\indent Then, by using the symmetric (resp. skew-symmetric) property of the shape operator $A$ (resp. the structure tensor field $\p$), we have
\begin{equation*}
-g(X,A\p Y)=2g(X,A\xh)g(\xt,Y)-2g(X,A\xt)g(\xh,Y)-g(Y,A\po X)
\end{equation*}
for any tangent vector fields $X$ and $Y$ on $M$.
Then it can be rewritten as below:
\begin{equation}\label{eq: 1.15}
A\p X=2\eh(X)A\xt-2\et(X)A\xh+A\po X.
\end{equation}

\begin{note}
{\rm Hereafter, the process used from \eqref{eq: 1.14} to \eqref{eq: 1.15} will be expressed as ``{\it taking a symmetric part of \eqref{eq: 1.14}}''.}
\end{note}

Bearing in mind that $\x=\xo\in \QP$, \eqref{eq: 1.2} is simplified:
\begin{equation}\label{eq: 1.16}
\begin{split}
SX&=(4m+7)X-7\e(X)\x-2\etw(X)\xtw\\
  &\quad-2\eh(X)\xth+\po\p X+h AX-A^2X.
\end{split}
\end{equation}

Multiplying $\po$ to \eqref{eq: 1.16} and using basic formulas in \cite[Section~$\rm 2$]{LCW},
we have
\begin{equation}\label{eq: 1.17}
\begin{split}
\po\p AX=2\eta_{3}(AX)\x_{3}+2\eta_{2}(AX)\x_{2}-AX+\al\e(X)\x.
\end{split}
\end{equation}

By replacing $X$ as $AX$ into \eqref{eq: 1.16} and using \eqref{eq: 1.17}, we obtain
\begin{equation}\label{eq: 1.18}
SAX=(4m+6)AX-6\al\e(X)\x+h A^2X-A^3X
\end{equation}
and taking a symmetric part of \eqref{eq: 1.18} again, we get
\begin{equation}\label{eq: 1.19}
ASX=(4m+6)AX-6\al\e(X)\x+h A^2X-A^3X.
\end{equation}
\indent Comparing \eqref{eq: 1.18} and \eqref{eq: 1.19}, we conclude that
\begin{equation*}
SAX=ASX
\end{equation*}
for any tangent $X$.
\end{proof}

By the way, we have equations \eqref{eq: 1.13} and \eqref{eq: 1.15} for the Ricci tensor likewise related to the shape operator.
We may consider similar ones about the Ricci tensor as below:

\begin{lemma}\label{lemma 1.3}
Let $M$ be a Hopf hypersurface in $\GBt$. If the Reeb vector field $\x$ belongs to $\QP$, we have the following formulas
\begin{enumerate}[\rm (i)]
\item {$\p SX=2\eta_{3}(SX)\x_{2}-2\eta_{2}(SX)\x_{3}+\p_{1} SX+\text{Rem}(X)$}\label{eq: i} and
\item {$S\p X=2\eh(X)S\xt-2\etw(X)S\xh+S\po X+\text{Rem}(X)$}, \label{eq: ii}
\end{enumerate}
 where the remainder term $\text{Rem}(X)$ is denoted by $\text{Rem}(X)=4(m+2)\{2\etw(X)\xth-2\eh(X)\xtw+\p X-\po X\}$.
\end{lemma}

\begin{proof}
Multiplying $\p$ to \eqref{eq: 1.16}, we get the equivalent equation of the Left side of \eqref{eq: i} as follows:

\begin{equation}\label{eq: 1.20}
\begin{split}
\p SX=(4m+7)\p X-\po X+2\etw(X)\xth-2\eh(X)\xtw+h\p AX-\p A^2X.
\end{split}
\end{equation}

%\begin{equation}\label{eq: 1.20}
%\begin{split}
%\p SX&=(4m+7)\p X-\po X+2\etw(X)\xth-2\eh(X)\xtw\\
%     &\quad+h\p AX-\p A^2X.
%\end{split}
%\end{equation}

Using \eqref{eq: 1.14}, and \eqref{eq: 1.15}, the right side of \eqref{eq: i} is can be replaced by
\begin{equation}\label{eq: 1.21}
\begin{split}
&\quad 2\e_{3}(SX)\x_{2}-2\e_{2}(SX)\x_{3}+\p_{1} SX+Rem(X)\\
&=2\eh\big((4m+7)X-2\eh(X)\xth+\po\p X+h AX-A^2X\big)\xtw\\
&\quad-2\etw\big((4m+7)X-2\etw(X)\xtw+\po\p X+h AX-A^2X\big)\xth\\
&\quad+(4m+7)\po X-2\etw(X)\xtw+2\eh(X)\xth-\p X+h\po AX-\po A^2X\\
&\quad+Rem(X)\\
&=(4m+7)\p X-\po X+2\etw(X)\xth-2\eh(X)\xtw+h\p AX-\p A^2X.
\end{split}
\end{equation}

Combining \eqref{eq: 1.20} and \eqref{eq: 1.21}, we get the equation \eqref{eq: i}. In addition, \eqref{eq: ii} can be obtained by taking a symmetric part of \eqref{eq: i}.
\end{proof}

By virtue of Lemmas~\ref{lemma 1.2} and \ref{lemma 1.3}, we assert the following:

\begin{lemma}\label{lemma 1.4}
Let $M$ be a Hopf hypersurface in
$\GBt$ with \eqref{C-1}. If $\x \in \QP$, we have $A(\p S-S\p)=(\p S-S\p)A$.
\end{lemma}

\begin{proof}
By \eqref{eq: i} (resp. \eqref{eq: ii}) in Lemma~\ref{lemma 1.3}, we have
the left side of \eqref{eq: 1.1} (the right side of \eqref{eq: 1.1}) as follows:

\begin{equation}\label{eq: 1.22}
\left \{
\begin{aligned}
\RXP SX=2\p SX+\al A\p SX+Rem(X), \\
S\RXP X=2S\p X+\al SA\p X+Rem(X).
\end{aligned}
\right.
\end{equation}

\indent Combining equations in \eqref{eq: 1.22}, we have
\begin{equation}\label{eq: 1.24}
\begin{split}
\RXP SX-S\RXP X=2\p SX+\al A\p SX-2S\p X-\al SA\p X=0.
\end{split}
\end{equation}
\vskip 5pt
\noindent {\bf Case 1\,:} \quad $\al = 0$. Equation \eqref{eq: 1.24} becomes $S\p X=\p SX$.
By virtue of \cite[Theorem]{S02}, we conclude that if $M$ is a Hopf hypersurface in complex two-plane Grassmannians $\GBt$ with \eqref{eq: 1.1}, then $M$ satisfies the condition of Type~$(A)$.
\vskip 5pt
Thus, we may assume the following case.
\vskip 5pt
\noindent {\bf Case 2\,:} \quad $\al\neq0$.
\vskip 5pt
Using Lemma~\ref{lemma 1.2}, \eqref{eq: 1.24} becomes
\begin{equation}\label{eq: 1.25}
\begin{split}
2(\p S-S\p )+\al(A\p S- AS\p)=0.
\end{split}
\end{equation}
\indent Taking a symmetric part of \eqref{eq: 1.25}, we have
\begin{equation}\label{eq: 1.26}
2(\p S-S\p)-\al(S\p A-\p SA)=0.
\end{equation}
\indent Combining \eqref{eq: 1.25} and \eqref{eq: 1.26}, we know
\begin{equation}\label{eq: 1.27}
A(\p S-S\p)=(\p S-S\p)A.
\end{equation}
\end{proof}

\begin{lemma}\label{lemma 1.5}
Let $M$ be a Hopf real hypersurface in $\GBt$. If $M$ satisfies $A(\p S-S\p)=(\p S-S\p)A$ and $\x \in \QP$, then we have $S\p=\p S$.
\end{lemma}

\begin{proof}
Since the shape operator $A$ and the tensor $\p S-S\p$ are both symmetric operators and commute with each other, they are diagonalizable. So there exists a common basis $\{E_1,E_2,...,E_{4m-1}\}$ such that the shape operator $A$ and the tensor $\p S-S\p$  both can be diagonalizable.
In other words, $AE_i=\lambda_i E_i$ and $(\p S-S\p )E_i=\beta_iE_i$, where $\lambda_i$ and $\beta_i$ are scalars for all $i\in{1,2,...,4m-1}$.

Here replacing $X$ by $\p X$ in \eqref{eq: 1.16} \big(resp. multiplying $\p$ to \eqref{eq: 1.16}\big), we have
%\begin{equation}\label{eq: 1.28}
%\begin{split}
%S\p X&=(4m+7)\p X-\po X+2\etw(X)\xth-2\eh(X)\xtw\\
%     &\quad+hA\p X-A^2\p X.
%\end{split}
%\end{equation}
%and
%\begin{equation}\label{eq: 1.28-1}
%\begin{split}
%\p SX&=(4m+7)\p X-\po X+2\etw(X)\xth-2\eh(X)\xtw\\
%     &\quad+h\p AX-\p A^2X.
%\end{split}
%\end{equation}

\begin{equation}\label{eq: 1.28}
\left \{
\begin{aligned}
S\p X=(4m+7)\p X-\po X+2\etw(X)\xth-2\eh(X)\xtw+hA\p X-A^2\p X, \\
\p SX=(4m+7)\p X-\po X+2\etw(X)\xth-2\eh(X)\xtw+h\p AX-\p A^2X.
\end{aligned}
\right.
\end{equation}

Combining equations in \eqref{eq: 1.28}, we get
\begin{equation}\label{eq: 1.29}
S\p X - \p SX = h A\p X - A^{2} \p X - h \p AX + \p A^{2}X.
\end{equation}
\indent Putting $X=E_i$ into \eqref{eq: 1.29} and using $AE_i=\lambda_i E_i$, we obtain
\begin{equation}\label{eq: 1.30}
(S\p-\p S)E_i  = h A\p E_i - A^{2} \p E_i - h \lambda_i\p E_i + \p \lambda_i^{2}E_i.
\end{equation}

Taking the inner product with $E_i$ into \eqref{eq: 1.30}, we have
\begin{equation*}
\begin{split}
\beta_i g(E_i,E_i)= h\lambda_i g(\p E_i,E_i ) - \lambda^{2}_{i}g(\p E_i,E_i)= 0.
\end{split}
\end{equation*}
\indent Since $g(E_i,E_i)\neq0$, $\beta_i=0$ for all $i\in{1,2,...,4m-1}$. This is equivalent to $(S\p-\p S)E_i=0$ for all $i\in{1,2,...,4m-1}$. It follows that $S\p X = \p SX$ for any tangent vector field $X$ on $M$.
\end{proof}

\vskip 5pt

Summing up Lemmas~\ref{lemma 1.2}, \ref{lemma 1.3}, \ref{lemma 1.4}, \ref{lemma 1.5} and \cite[Theorem]{S02}, we conclude that if $M$ is a Hopf hypersurface in complex two-plane Grassmannians $\GBt$ satisfying \eqref{C-1}, then $M$ satisfies the condition of Type~$(A)$.
\vspace {0.15in}

Hereafter, let us check whether the Ricci tensor of a model space of Type~$(A)$ satisfies the commuting condition \eqref{C-1}.

\vskip 5pt

%In the case of $\x\in\QP$, by using Lemmas~\ref{lemma 1.2}, \ref{lemma 1.3}, \ref{lemma 1.4} and \ref{lemma 1.5}, we can easily check that the commuting condition $S\p=\p S$ is equivalent condition to $(\RXP) S = S %(\RXP)$. Therefore by Lemma~\ref{lemma 1.5} and Theorem~$\rm B$, we can assert that:
%\begin{remark}\label{remark type A-1}
%\rm Real hypersurfaces of Type~$(A)$ in $\GBt$ satisfies the condition \eqref{C-1}.
%\end{remark}

From~\eqref{eq: 1.2} and \cite[Proposition~$\rm 3$]{BS1}, we obtain the following equations:
\begin{equation*}
SX = \left\{ \begin{array}{ll}
                (4m+h\al-\al^2)\x  & \mbox{if}\ \  X=\x \in T_{\al}\\
                (4m+6+h\be-\be^2)\x_{\nu}  & \mbox{if}\ \  X=\x_{\nu} \in T_{\beta} \\
                (4m+6+h\la-\la^2) X  & \mbox{if}\ \  X \in T_{\lambda}\\
                (4m+8) X      & \mbox{if}\ \  X \in T_{\mu}\\
\end{array}\right.
\end{equation*}

\begin{equation*}
\RX (X) = \left\{ \begin{array}{ll}
                0                                       & \mbox{if}\ \  X=\x \in T_{\al}\\
                (\al\be+2)\x_{\nu} & \mbox{if}\ \  X=\x_{\nu} \in T_{\beta} \\
                (\al\la+2)\p X & \mbox{if}\ \  X \in T_{\lambda}\\
                0    & \mbox{if}\ \  X \in T_{\mu}\\
\end{array}\right.
\end{equation*}

\begin{equation*}
(\RXP) X = \left\{ \begin{array}{ll}
                0                                       & \mbox{if}\ \  X=\x \in T_{\al}\\
                (\al\be+2)\p\x_{\nu} & \mbox{if}\ \  X=\x_{\ell} \in T_{\beta} \\
                (\al\la+2)\p X & \mbox{if}\ \  X \in T_{\lambda}\\
                0    & \mbox{if}\ \  X \in T_{\mu}.\\
\end{array}\right.
\end{equation*}

%Consider $X=\p\x_{\ell}$ into \eqref{C-1}.

Combining above three formulas, it follows that
\begin{equation*}
(\RXP) SX- S(\RXP) X = \left\{ \begin{array}{ll}
                0 & \mbox{if}\ \  X=\x \in T_{\al}\\
                0 & \mbox{if}\ \  X=\x_{\ell} \in T_{\beta} \\
                0 & \mbox{if}\ \  X \in T_{\lambda}\\
                0 & \mbox{if}\ \  X \in T_{\mu}.\\
\end{array}\right.
\end{equation*}

\begin{remark}\label{remark type A-3}
\rm When $\x \in \QP$, a Hopf hypersurface $M$ in $\GBt$ with \eqref{C-1} is locally congruent to of Type~$(A)$ by virtue of \cite[Theorem]{S02}.
\vskip 5pt.
\end{remark}

When $\x \in \Q$, a Hopf hypersurface $M$ in $\GBt$ with \eqref{C-1} is locally congruent to of Type~$(B)$ by virtue of \cite[Main Theorem]{LS}.

Now let us consider our problem for a model space of Type~$(B)$ which will be denoted by $M_{B}$. In order to do this, let us calculate $(\RXP) S = S \RXP$ related to the $M_{B}$. On $T_{x}M_{B}$, $x \in M_{B}$,
the equations \eqref{eq: 1.2} and \eqref{eq: 1.3} are reduced to the following equations, respectively:
\begin{align}
SX & = (4m+7)X - 3\eta(X)\x + hAX - A^2 X - \sum_{\nu =1}^3 \{ 3 \eta_{\nu}(X)\x_{\nu} +\eta(\p_{\nu}X)\p_{\nu}\x \}
    \ \  \text{and}\label{eq: 1.31} \\
& R_{\x}(X)= X - {\eta}(X){\x} +{\alpha}AX-\al^2{\eta}(X){\x}- {\SN}\Big\{{\EN}(X){\XN}+3{\EN}(\p X){\PNK}\Big\}.\label{eq: 1.32}
\end{align}
\indent From \eqref{eq: 1.31} and \eqref{eq: 1.31} and \cite[Proposition~$\rm 2$]{BS1}, we obtain the following
\begin{equation}\label{eq: 1.33}
SX = \left\{ \begin{array}{ll}
                (4m+4+h\al-\al^2)\x  & \mbox{if}\ \  X=\x \in T_{\al}\\
                (4m+4+h\be-\be^2)\x_{\ell}  & \mbox{if}\ \  X=\x_{\ell} \in T_{\beta} \\
                (4m+8)\p\x_{\ell}  & \mbox{if}\ \ X=\p \x_{\ell} \in T_{\gamma}\\
                (4m+7+h\la-\la^2) X  & \mbox{if}\ \  X \in T_{\lambda}\\
                (4m+7+h\mu-\mu^2) X      & \mbox{if}\ \  X \in T_{\mu}\\
\end{array}\right.
\end{equation}

\begin{equation}\label{eq: 1.34}
\RX (X) = \left\{ \begin{array}{ll}
                0                                       & \mbox{if}\ \  X=\x \in T_{\al}\\
                \al\be\x_{\ell} & \mbox{if}\ \  X=\x_{\ell} \in T_{\beta} \\
               4\p\x_{\ell} & \mbox{if}\ \ X=\p \x_{\ell} \in T_{\gamma}\\
                (1+\al\la)\p X & \mbox{if}\ \  X \in T_{\lambda}\\
                (1+\al\mu)\p X     & \mbox{if}\ \  X \in T_{\mu}\\
\end{array}\right.
\end{equation}

\begin{equation}\label{eq: 1.35}
(\RXP) X = \left\{ \begin{array}{ll}
                0                                       & \mbox{if}\ \  X=\x \in T_{\al}\\
                4\p\x_{\ell} & \mbox{if}\ \  X=\x_{\ell} \in T_{\beta} \\
                -\al\be\x_{\ell} & \mbox{if}\ \ X=\p \x_{\ell} \in T_{\gamma}\\
                (1+\al\mu)\p X & \mbox{if}\ \  X \in T_{\lambda}\\
                (1+\al\la)\p X     & \mbox{if}\ \  X \in T_{\mu}.\\
\end{array}\right.
\end{equation}

%Consider $X=\p\x_{\ell}$ into \eqref{C-1}.

From \eqref{eq: 1.33}, \eqref{eq: 1.34} and \eqref{eq: 1.35}, it follows that
\begin{equation}\label{eq: 1.36}
(\RXP) SX- S\RXP X = \left\{ \begin{array}{ll}
                0 & \mbox{if}\ \  X=\x \in T_{\al}\\
                4(h\be-\be^2-4)\p\x_{\ell} & \mbox{if}\ \  X=\x_{\ell} \in T_{\beta} \\
                \al\be(h\be-\be^2-4)\x_{\ell} & \mbox{if}\ \ X=\p \x_{\ell} \in T_{\gamma}\\
                (1+\al\mu)(\lambda-\mu)(h-\lambda-\mu)\p X  & \mbox{if}\ \  X \in T_{\lambda}\\
                (1+\al\lambda)(\mu-\lambda)(h-\lambda-\mu)\p X      & \mbox{if}\ \  X \in T_{\mu}.\\
\end{array}\right.
\end{equation}
\indent By calculation, we have $\lambda+\mu=\be$ on $M_{B}$.
From \eqref{eq: 1.36}, we see that $M_{B}$ satisfies \eqref{C-1}, only when $h=\be$ and $h\be-\be^2-4=0$. This gives us to a contradiction.
\vskip 5pt
Hence, we give a complete proof of Theorem~$1$.

%%%%%%%%%%%%%%%%%%%%%%%%%%%%%%%%% 3333333333333 %%%%%%%%%%%%%%%%%%%%%%%%%%%%%%%%%%%%
\section{Proof of Theorem~$2$}\label{Section 2}
\setcounter{equation}{0}
\renewcommand{\theequation}{2.\arabic{equation}}
\vspace{0.13in}
For a commuting problem in quaternionic space forms Berndt \cite{Ber01} has
introduced the notion of normal Jacobi operator $\bar{R}(X,N)N \in T_x M$, $x \in M$ for real hypersurfaces $M$ in quaternionic
projective space $\Bbb{Q}P^m$ or in quaternionic hyperbolic space $\Bbb{Q}H^m$, where $\bar{R}$ denotes the curvature tensor of  $\Bbb{Q}P^m$ or of $\Bbb{Q}H^m$. He \cite{Ber01} has also shown that the
curvature adaptedness, when the normal Jacobi operator commutes
the shape operator $A$, is equivalent to the fact that the
distributions $\Q$ and $\QP=\text{Span}\{\xi_1,\xi_2 ,\xi_3 \}$ are invariant by the shape operator $A$ of $M$,
where $T_x M=\Q\oplus \QP$, $x \in M$. In this section, by using the notion of normal Jacobi operator $\bar{R}(X,N)N\in T_x M$, $x \in M$ for real hypersurfaces $M$ in $\GBt$ and geometric quantities in \cite{PJS} and \cite{S02}, we give a complete proof of Theorem~$2$.

From now on, let $M$ be a Hopf hypersurface in $\GBt$ with
\begin{equation}\label{eq: 2.1}
(\RNP) SX = S (\RNP)X
\end{equation}
for any tangent vector field $X$ on $M$.
The normal Jacobi operator $\Bar{R}_{N}$ of $M$ is defined by $\RNX=\Bar{R}(X,N)N$ for any tangent vector $X \in T_{x}M$, $x \in M$.
In \cite[Introduction]{PJS}, we obtain the following equation
\begin{equation}\label{eq: 2.2}
\begin{split}
\RNX&=X+3\e(X)\x+3\SN\EN(X)\XN\\
    &\quad-\SN\{\ENx\PN\PX-\ENx\e(X)\XN-\EN(\PX)\PN\x \}.
\end{split}
\end{equation}

\begin{lemma}\label{lemma 2.1}
Let $M$ be a Hopf hypersurface in $\GBt$, $m \geq 3$, with \eqref{C-2}. If the principal curvature $\alpha=g(A\x, \x)$ is constant along the direction of $\x$, then $\x$ belongs to either the distribution $\Q$ or the distribution~$\QP$.
\end{lemma}
\begin{proof}
%Put $\x$ satisfying \eqref{**} with the same assumptions as the proof of lemma~\ref{lemma 1.1}.
In order to prove this lemma, we assume \eqref{**} again,
%\begin{equation}\label{**}
%\x = \eta(X_{0})X_{0}+\eta(\x_{1})\x_{1}
%\end{equation}
for some unit vectors $X_{0} \in \Q$, $\x_{1} \in \QP$ and $\eta(X_{0})
\eta(\x_{1})\neq 0$.

On the other hand, from \eqref{eq: 2.2} and \eqref{**}, we have
\begin{align}
& \RN \X=4\e^2(\X)\X+4\eo(\x)\e(\X)\xo\ \  \text{and}\label{eq: 2.3} \\
& \RN \x=4\x+4\eo(\x)\xo.\label{eq: 2.4}
\end{align}
% and \eqref{eq: 1.8}
Using \eqref{eq: 1.7}, \eqref{eq: 1.8}, \eqref{eq: 2.3}, \eqref{eq: 2.4} and inserting $X=\p\X$ into \eqref{eq: 2.1}, we have the following equations:
\begin{equation}\label{eq: 2.5}
\begin{split}
\text{the left side of \eqref{eq: 2.1}}&=(\RNP)S\p\X=\si\RNP^2\X\\
    &=-\si\RN \X+\si\e(\X)\RN \x\\
    &=-\si\{4\e^2(\X)\X+4\eo(\x)\e(\X)\xo \}\\
    &\quad+\si\{4\e(\X)\x+4\e(\X)\eo(\x)\xo \}\\
    &=4\si\e(\X)\eo(\x)\xo
\end{split}
\end{equation}
\begin{equation}\label{eq: 2.6}
\begin{split}
\text{the right side of \eqref{eq: 2.1}}&=S\RN (\p^2\X)=-S\RN \X+\e(\X)S\RN \x\\
    &=-4\e^2(\X)S\X-4\e(\x)\e(\X)S\xo\\
    &\quad+4\e(\X)S\x+4\e(\X)\e(\xo)S\xo\\
    &=-4\e^2(\X)\big\{(4m+7+\al h-\al^2)\X-3\e(\X)\x\\&\quad+\eo^2(\x)\X-\e(\X)\eo(\x)\xo\big\}\\
    &\quad+4\e(\X)\big\{(4m+4+\al h-\al^2)\x-4\eo(\x)\x\big\},
\end{split}
\end{equation}
where $\sigma:=4m+8+h\ka+\ka^2$.
Recalling that $\e(\X)\neq 0$ and combining \eqref{eq: 2.5} and \eqref{eq: 2.6}, we have
\begin{equation*}
\begin{split}
4\si\e(\X)\eo(\x)\xo
&=-4\e^2(\X)\big\{(4m+7+\al h-\al^2)\X-3\e(\X)\x\\&\quad+\eo^2(\x)\X-\e(\X)\eo(\x)\xo\big\}\\
&\quad+4\e(\X)\big\{(4m+4+\al h-\al^2)\x-4\eo(\x)\x\big\}.
\end{split}
\end{equation*}
\indent Taking the inner product of above equation with $\X$, we get
\begin{equation*}
\begin{split}
0&=-4\e(\X)\big\{(4m+7+\al h-\al^2)-3\e^2(\X)+\eo^2(\x)\big\}\\
 &\quad+4\{(4m+4+\al h-\al^2)\e(\X)\}\\
 &=-4\e(\X)\{3-3\e^2(\x)+\eo^2(\x)\}\\
 &=-16\e(\X)\eo^2(\x).
\end{split}
\end{equation*}
\indent This gives a contradiction. Thus, we give a complete proof of this lemma.
%So we shall divide our consideration in two cases that $\x$
%belongs to either the distribution $\QP$ or $\Q$,
%respectively.
\end{proof}

Now this case implies that $\x$ belongs to the distribution $\QP$.
\vskip 5pt

\begin{lemma}\label{lemma 2.2}
Let $M$ be a Hopf hypersurface in
$\GBt$ with \eqref{eq: 2.1}. If $\x \in \QP$, we have $S\p=\p S$.
\end{lemma}

\begin{proof}
Putting $\x=\xo \in \QP$ for our convenience sake, \eqref{eq: 2.2} becomes
\begin{equation*}
\begin{split}
\RNX=X+7\E(X)\x+2\etw(X)\xtw+2\eh(X)\xth-\po\p X.
\end{split}
\end{equation*}

Because of \eqref{eq: i} and \eqref{eq: ii} in lemma~\ref{lemma 1.3}, we have the following equations:

\begin{equation}\label{eq: 2.7}
\left \{
\begin{aligned}
\RN \p SX=2\p SX-\text{Rem}(X),\\
S\RN \p X=2S\p X-\text{Rem}(X),
\end{aligned}
\right.
\end{equation}
where $\text{Rem}(X)=4(m+2)\{2\etw(X)\xth-2\eh(X)\xtw+\p X-\po X\}$.\\
\indent Combining equations in \eqref{eq: 2.7}, we conclude that \eqref{eq: 2.1} is equivalent to $S\p X=\p SX$ .
\end{proof}

\vskip 5pt

In the case of $\x\in\QP$, by using \eqref{eq: i} and \eqref{eq: ii} in Lemma~\ref{lemma 1.3}, and Lemma~\ref{lemma 2.2},
we can be easily seen that the commuting condition $S\p=\p S$ is equivalent to $(\RNP) S = S (\RNP)$.
\vskip 5pt
\noindent Therefore, by Lemma~\ref{lemma 2.2} and \cite[Theorem]{S02}, we can assert that:
\begin{remark}\label{remark type A-4}
\rm Real hypersurfaces of Type~$(A)$ in $\GBt$ satisfies the condition \eqref{C-2}.
\end{remark}

\vskip 5pt
When $\x \in \Q$, a Hopf hypersurface $M$ in $\GBt$ with \eqref{C-2} is locally congruent to of Type~$(B)$ by virtue of \cite[Main Theorem]{LS}.

\vskip 5pt
Let us consider our problem for a model space of Type~$(B)$ which will be denoted by $M_{B}$. In order to do this, let us calculate $(\RNP) S = S (\RNP)$ of $M_{B}$.
From \cite[Proposition~$\rm 2$]{BS1}, we obtain
\begin{equation}\label{eq: 2.8}
\RNX = \left\{ \begin{array}{ll}
                4\x                                      & \mbox{if}\ \  X=\x \in T_{\al}\\
                4\x_{\ell} & \mbox{if}\ \  X=\x_{\ell} \in T_{\beta} \\
                0 & \mbox{if}\ \ X=\p \x_{\ell} \in T_{\gamma}\\
                X & \mbox{if}\ \  X \in T_{\lambda}\\
                X & \mbox{if}\ \  X \in T_{\mu},\\
\end{array}\right.
\end{equation}

\begin{equation}\label{eq: 2.9}
(\RNP) X = \left\{ \begin{array}{ll}
                0                                       & \mbox{if}\ \  X=\x \in T_{\al}\\
                0 & \mbox{if}\ \  X=\x_{\ell} \in T_{\beta} \\
                -4\x_{\ell} & \mbox{if}\ \ X=\p \x_{\ell} \in T_{\gamma}\\
                \p X & \mbox{if}\ \  X \in T_{\lambda}\\
                \p X     & \mbox{if}\ \  X \in T_{\mu}.\\
\end{array}\right.
\end{equation}

From \eqref{eq: 2.8} and \eqref{eq: 2.9}, it follows that
\begin{equation*}
(\RNP) SX- S(\RNP) X = \left\{ \begin{array}{ll}
                0 & \mbox{if}\ \  X=\x \in T_{\al}\\
                0 & \mbox{if}\ \  X=\x_{\ell} \in T_{\beta} \\
                4(h\be-\be^2-4)\x_{\ell} & \mbox{if}\ \ X=\p \x_{\ell} \in T_{\gamma}\\
                (\lambda-\mu)(h-\lambda-\mu)\p X  & \mbox{if}\ \  X \in T_{\lambda}\\
                (\mu-\lambda)(h-\lambda-\mu)\p X  & \mbox{if}\ \  X \in T_{\mu}.\\
\end{array}\right.
\end{equation*}
\indent We see that $M_{B}$ satisfies \eqref{C-2}, only when $h=\be$ and $h\be-\be^2-4=0$. This gives us to a contradiction.

\vskip 5pt

Thus, we can give a complete proof of Theorem~$2$ in the introduction.

\vskip 5pt
%%%%%%%%%% Reference %%%%%%%%%%%%%%%%%%%%%%%%%%%%%%%%%%%%%%%%%%%%%%%%%%%%%%%%%%%%%%%%%%%%%%%%%%%%%%%%%%%%%%%%%%%%%%%%%%%%%%%%%%

\end{document}